\documentclass[12pt]{amsart}
\usepackage{amsmath}
\usepackage{amsthm}
\usepackage{amssymb}
\usepackage{amscd}
\usepackage[colorlinks]{hyperref}

\newcommand{\gothic}{\mathfrak}

\newcommand{\m}{{\gothic{m}}}

\newcommand{\8}{{\infty}}

\newcommand{\depth}{\operatorname{depth}}

\newcommand{\im}{\operatorname{im}}
\newcommand{\pd}{\operatorname{pd}}

\newcommand{\syz}{\operatorname{syz}}
\newcommand{\Tor}{\operatorname{Tor}}

\newcommand{\al}{\alpha}
\newcommand{\brq}{^{[q]}}

\newcommand{\inc}{\subseteq}

\newcommand{\length}{\lambda}

\newcommand{\dt}{\bullet}

\renewcommand{\hat}{\widehat}
\renewcommand{\bar}{\overline}
\renewcommand{\phi}{\varphi}

\theoremstyle{plain}
\newtheorem{thm}{Theorem}
\newtheorem{cor}[thm]{Corollary}
\newtheorem{prop}[thm]{Proposition}
\newtheorem{lemma}[thm]{Lemma}

\theoremstyle{definition}
\newtheorem{defn}[thm]{Definition}

\theoremstyle{remark}
\newtheorem{rmk}[thm]{Remark}

\begin{document}
\title[Asymtotic vanishing conditions]
{Asymptotic vanishing conditions which force regularity in local rings
of prime characteristic}
\author{Ian M. Aberbach}
\address{Mathematics Department \\
    University of Missouri\\
    Columbia, MO 65211 USA}
\email{aberbach@math.missouri.edu}
\urladdr{http://www.math.missouri.edu/$\sim$aberbach}

\author{Jinjia Li}
\address{Department of Mathematics \\
        Syracuse University  \\
         Syracuse, NY 13244-1150  USA}
\email{jli32@syr.edu} \urladdr{http://web.syr.edu/$\sim$jli32}

\date{\today}
\bibliographystyle{amsplain}

\numberwithin{thm}{section}
\numberwithin{equation}{section}

\begin{abstract} Let $(R,\m,k)$ be a local (Noetherian) ring of
positive prime characteristic $p$ and dimension $d$.  Let $G_\dt$
be a minimal resolution of the residue field $k$, and for each
$i\ge 0$, let $\gothic t_i(R) = \lim_{e\to \8}
{\length(H_i(F^e(G_\dt)))}/{p^{ed}}$.  We show that if $\gothic
t_i(R) = 0$ for some $i>0$, then $R$ is a regular local ring.
Using the same method, we are also able to show that if $R$ is an
excellent local domain and $\Tor_i^R(k,R^+) = 0$ for some $i>0$,
then $R$ is regular (where $R^+$ is the absolute integral closure
of $R$). Both of the two results were previously known only for $i
= 1$ or $2$ via completely different methods.
\end{abstract}
\maketitle
\section{Introduction}\label{intro}
Throughout this paper, we assume $(R, \m, k)$ is a commutative
local Noetherian ring of characteristic $p>0$ and dimension $d$.
The Frobenius endomorphism $f_R:R \to R$ is defined by
$f_R(r)=r^p$ for $r \in R$. Each iteration $f_R^e$ defines a new
$R$-module structure on $R$, denoted ${}^{f^e}\!\! R$, for which
$a\cdot b=a^{p^e}b$. For any $R$-module $M$, $F_R^e(M)$ stands for
$M\otimes_R {}^{f^e}\!\! R$, the $R$-module structure of which is
given by base change along the Frobenius endomorphism. When $M$ is
a cyclic module $R/I$, it is easy to show that $F^e(R/I) \cong
R/I^{[p^e]}$, where $I^{[p^e]}$ denotes the ideal generated by the
$p^e$-th power of the generators of $I$.

In the sequel, $\length(-)$ denotes the length function. $q$
usually denotes a varying power $p^e$.

In \cite{Li}, the second author introduced the following higher
$\Tor$ counterparts for the Hilbert-Kunz multiplicity
\begin{equation} \label{def}
\gothic t_i(R)= \underset{q \to \8}{\lim}
\length(\Tor_i(k,{}^{f^e}\!\! R))/q^d \end{equation} and showed
that $R$ is regular if and only if $\gothic t_i(R)=0$ for $i=1$ or
$i=2$. In another paper \cite{Ab2}, (which extends results
obtained in \cite{Sch}), the first author proved that an excellent
local domain $R$ is regular if and only if $\Tor_1^R(k,R^+)=0$,
where $R^+$ is the absolute integral closure of $R$ (i.e, the
integral closure of $R$ in the algebraic closure of its field of
fraction). It is not difficult to see that this is also equivalent
to $\Tor_2^R(k,R^+)=0$. In fact, since $R^+$ is a big
Cohen-Macaulay algebra \cite{HL}, by Proposition~2.5 of
\cite{Sch}, the condition $\Tor_2^R(k,R^+)=0$ forces $R$ to be
Cohen-Macaulay. Therefore $\Tor_1^R(R/(\bold x), R^+)=0$ for any
system of parameters $\bold x$ for $R$. $R/(\bold x)$ has a
filtration by $k$. Tensoring the short exact sequences in the
filtration by $R^+$, from the resulting long exact sequences and
the condition $\Tor_2^R(k, R^+)=0$, one gets $\Tor_1^R(k, R^+)=0$.
The methods used in \cite{Li} and \cite{Ab2} are completely
different. However, neither of them works for the case $i \geq 3$
(unless $R$ is assumed to be Cohen-Macaulay).

The main result of this paper, Theorem \ref{mainprop} below,
states that over an equidimensional complete local ring $(R,\m,
k)$ of prime characteristic, if a finitely generated free
resolution of $k$ has \emph{stably phantom homology} at the $i$-th
spot for some $i>0$, then $R$ is regular. As we observe in
Proposition \ref{compareReq}, under mild conditions, a complex of
finitely generated free modules with finite length homology has
stably phantom homology at some spot if and only if a similar
asymptotic length as defined in (\ref{def}) vanishes at the same
spot. Consequently, Theorem \ref{mainprop} allows us to
simultaneously extend the main results in \cite{Li} and \cite{Ab2}
to all $i>0$. Specifically, we have

\vspace{2.5mm}

\noindent {\bf{Corollary}~\ref{mainthm}} Let $R$ be a local ring
of characteristic $p$.  If $\gothic t_i(R) = 0$ for some $i>0$
then $R$ is regular.

\vspace{2.5mm}

\noindent {\bf{Corollary}~\ref{rplus}} Let $(R,\m, k)$ be an
excellent local domain. If $\Tor_i^R(k, R^+) = 0$ for some $i>0$
then $R$ is regular.

\vspace{2.5mm}

As a further consequence of Corollary \ref{mainthm}, we give an
extended version (in the prime characteristic case) of a criterion
of regular local rings due to Bridgeland and Iyengar \cite{BI}.

\section{Asymptotic vanishing and stably phantom homology}\label{background}
We recall some basic definitions in tight closure theory. The
reader should refer to \cite{HH} for details. Let $R^o$ denote the
complement of the union of the minimal primes of $R$. Let $N
\subseteq M$ be two $R$-modules. We write $N\brq_M$ for
$\ker(F^e(M) \to F^e(M/N))$. We say that $x \in M$ is in the
\emph{tight closure} of $N$, $N^*$, if there exist $c \in R^o$ and
an integer $q_0$ such that for all $q \geq q_0$ , $cx^q \in
N\brq_M$.

We say that $c\in R^o$ is a $q_0$-\emph{weak test element} (or
simply a weak test element) if for every finitely generated module
$M$ and submodule $N$, $x\in M$ is in $N^*$ if and only if $cx^q
\in N\brq_M$ for all $q \geq q_0$. If this holds with $q_0 = 1$,we
call $c$ a \emph{test element}.

Let $(G_\dt,\partial_\dt)$ be a complex over $R$. We say that
$G_\dt$ has \emph{phantom homology} at the $i$th spot (or simply,
$G_\dt$ is phantom at the $i$th spot), if $\ker
\partial_i \subseteq (\im \partial_{i+1})^*_{G_i}$. We say that $G_\dt$ has
\emph{stably phantom homology} at the $i$th spot (or simply,
$G_\dt$ is stably phantom at the $i$th spot) if $F^e(G_\dt)$ has
phantom homology at the $i$th spot for all $e \geq 0$.

We say a complex $G_\dt$ is a \emph{left complex} if $G_i=0$ for
$i<0$. A left complex $G_\dt$ of projective modules is call a
\emph{phantom resolution} of $M$ if $H_0(G_\dt)\cong M$ and
$G_\dt$ has phantom homology for all $i>0$. If such a complex
$G_\dt$ is finite, we say $G_\dt$ is a \emph{finite phantom
resolution} of $M$.

By a \emph{resolution} of a module $M$, we always mean a
resolution of $M$ by finitely generated free modules. It is easy
to check that for a given module $M$, whether a resolution of $M$
has stably phantom homology at the $i$th spot is independent of
the choice of the resolution. We will use this fact many times
without explicitly mentioning it.

The following two lemmas are easy consequences of the above
definitions, we leave them for the reader to verify.
\begin{lemma} Let $G_\dt$ be a complex of finitely generated $R$-modules.
\begin{enumerate}
\item If for some $d\in R^o$ ($d$ need not  be a test element),
$dH_i(F^e_{R}(G_\dt))=0$ for all $e \geq 0$, then $G_\dt$ has
stably phantom homology at the $i$th spot. \item Suppose $R$
admits a test element $c$. If $G_\dt$ has phantom homology at the
$i$th spot for some $i$, then $cH_i(G_\dt)=0$. In particular, if
$G_\dt$ has stably phantom homology at the $i$th spot for some
$i$, then $cH_i(F^e(G_\dt))=0$ for all $e \geq 0$.
\end{enumerate}
\end{lemma}
\begin{lemma} \label{cone}Suppose $R$ admits a test element. Let $\alpha: F_\dt \to G_\dt$ be a chain map between
two complexes $F_\dt$ and $G_\dt$ over $R$. If both $F_\dt$ and
the mapping cone of $\alpha$ have stably phantom homology at the
$i$th spot, then so does $G_\dt$.
\end{lemma}
We will make ample use of one of the main results of Seibert's
paper \cite{Sei}:
\begin{prop}[\cite{Sei}, Propostion 1]\label{prop:seibert}
Let $R$ be a local ring of characteristic $p$.  Let $G_\dt$ be a
left complex of finitely generated free $R$-modules, and $N$ a finitely
generated $R$ module such that $G_\dt \otimes_R N$ has homology of
finite length.  Let $t = \dim N$.  Then for each $i\ge 0$ there
exists $c_i\in \mathbb{R}$ such that
\begin{equation*}
\length\left(\left(H_i(F^e_R(G_\dt)\otimes_R N\right)\right)
 = c_i q^t + O(q^{t-1}).
\end{equation*}
\end{prop}

We need a version of \cite{HH}, Theorem 8.17 for equidimensional rings that
are not reduced.
\begin{lemma}\label{equidimllength}
Let $(R,\m)$ be a complete equidimensional local ring of dimension $d$.
Suppose that $N \inc W \inc M$ are finitely generated modules such that
$W/N$ has finite length.  Then $W \inc N^*_M$ if and only if
$\lim_{q\to\8}\length(W\brq_M/N\brq_M)/q^d  =0$.
\end{lemma}
\begin{proof}
The implication $\Rightarrow$ follows from \cite{HH}, Theorem 8.17(a).

Let $J = \sqrt 0$ be the nilradical, and let $\bar R = R/J$.  By mapping
a free module onto $M$ and taking preimages, there is no loss of
generality in assuming that $M$ is free.  Clearly, for each $q$,
$\length\left(\dfrac{W\brq_M + JM}{N\brq_M +JM}\right) \le
\length\left(\dfrac{W\brq_M }{N\brq_M}\right)$, so by \cite{HH}, Theorem8.17(b),
$W+JM$ is in the tight closure of $N+JM$ computed over $\bar R$.
This implies that $W \inc N^*_M$ when computed over $R$
(see \cite{HH}, Proposition 8.5(j)).  This shows $\Leftarrow$.
\end{proof}

\begin{defn}  Let $R$ be a local ring of dimension $d$.
Let $I$ be the intersection of the primary components of $(0)$ which
have dimension $d$.  Then $R^{eq} := R/I$.
\end{defn}

\begin{prop}\label{compareReq}
Let $R$ be a local ring of dimension $d$, and $G_\dt$ a left
complex of finitely generated free modules with finite length
homology.  Then for $i\ge 0$
\begin{equation*}
\lim_{q\to\8} \dfrac{\length\left(H_i(F^e_R(G_\dt))\right)}{q^d} =
\lim_{q\to\8} \dfrac{\length\left(H_i(F^e_{R^{eq}}(G_\dt\otimes_R
R^{eq}))\right)}{q^d}.
\end{equation*}
In the case that either is $0$ (in which case both are),
 $G_\dt\otimes_R R^{eq}$ has stably phantom homology
at the $i$th spot (over $R^{eq}$). The converse is also true when
$R^{eq}$ admits a test element $c$ that is regular on $R^{eq}$
(e.g, $R$ is excellent, reduced and unmixed).
\end{prop}
\begin{proof}  Let $I$ be as in the definition of $R^{eq}$.
Let $J$ be the intersection
of primary components of  $(0)$ of dimension strictly less than $d$,
so $(0) = I \cap J$ and $\dim(R/J) < d$.  There is then an exact sequence
\begin{equation}\label{eqn:decompose}
0 \to R \to R^{eq} \oplus R/J \to R/(I+J) \to 0
\end{equation}
 (and we note that $\dim(R/(I+J))
< d$).

By Proposition~\ref{prop:seibert}, for all $j$,
\begin{equation}\label{eqn:smallhomology}
\lim_{q\to\8}
\dfrac{\length\left(H_j\left(F^e(G_\dt)\otimes_R R/J\right)\right)}{q^d}
=\lim_{q\to\8}
\dfrac{\length\left(H_j\left(F^e(G_\dt)\otimes_R R/(I+J)\right)\right)}{q^d}
 = 0.
\end{equation}
If we tensor the short exact sequence~\ref{eqn:decompose} with
$F^e(G_\dt)$, we get a short exact sequence of complexes. The
resulting long exact sequence in homology and the observations in
equation(~\ref{eqn:smallhomology}) give the desired equation.

Suppose now that the given limit is $0$. Without loss of
generality, we assume $R=R^{eq}$. Let $\partial_i$ denote the map
from $G_i$ to $G_{i-1}$ in the complex $G_\dt$. It is easy to
check that $(\ker{\partial_i})^{[q]}_{G_i} \subseteq
\ker(\partial_i^{[q]})$ and $(\im{\partial_{i+1}})^{[q]}_{G_i} =
\im(\partial_{i+1}^{[q]})$. So Lemma~\ref{equidimllength} shows
that $G_\dt$ has stably phantom homology at the $i$th spot over
$R^{eq}$.

Conversely, assume $R$ is equidimensional with a test element $c$.
Suppose $G_\dt$ has stably phantom homology at the $i$th spot.
Then $cH_i(F^e_{R}(G_\dt))=0$ for all $e \geq 0$. Since $c$ is
regular on $R$, one has an embedding $H_i(F^e_R(G_\dt))\otimes
R/cR \hookrightarrow H_i(F^e_{R/cR}(G_\dt\otimes (R/cR)))$, i.e.,
$H_i(F^e_{R}(G_\dt)) \hookrightarrow H_i(F^e_{R/cR}(G_\dt\otimes
(R/cR)))$ for all $e \geq 0$. Thus
\begin{equation}
\lim_{q\to\8} \dfrac{\length\left(H_i(F^e_R(G_\dt))\right)}{q^d}
\leq \lim_{q\to\8} \dfrac{\length\left(H_i(F^e_{R/cR}(G_\dt\otimes
(R/cR)))\right)}{q^d} =0
\end{equation}
The equality on the right hand side follows from Proposition
\ref{prop:seibert} again.
\end{proof}

\section{The main theorem}\label{maintheorem}
\begin{thm}\label{mainprop}
Let $(R,\m, k)$ be an excellent local ring of characteristic $p$.
Let $G_\dt$ be a resolution of $k$.  If for some $i>0$ the complex
$G_\dt\otimes_R R^{eq}$ is stably phantom at the $i$th spot, then
$R$ is regular.
\end{thm}

\begin{proof}
Write $(0) = I \cap J$ as in Proposition~\ref{compareReq}. Let
$c\in R$ be an element which is a $q_0$-weak test element in
$R^{eq}$, and let $c_1 \in J$ but not in any minimal prime of $I$.
Set $d = cc_1$. We claim that if $M$ is any module of finite
length, $(P_\dt, \beta_\dt)$ is a resolution for $M$, and $w \in
\ker(\beta_i\brq)$, then $dw^{q_0} \in \im(\beta_{i+1}^{[qq_0]})$.
We can induce on the length of the module $M$, with $\length(M)=1$
known.  If $\length(M)>1$ take a short exact sequence $0 \to M_1
\to M \to M_2 \to 0$ with $1 \le \length(M_1) < \length(M)$.  Let
$G_\dt$ be a resolution of $M_1$. Then there is a chain map $G_\dt
\to P_\dt$, and the mapping cone $T_\dt$ is a resolution of $M_2$.
By the induction hypothesis and the observation above, $G_\dt
\otimes_R R^{eq}$ and $T_\dt\otimes_R R^{eq}$ are stably phantom
over $R^{eq}$ at the $i$th spot. Then Lemma \ref{cone} shows
$P_\dt \otimes_R R^{eq}$ is  stably phantom over $R^{eq}$ at the
$i$th spot. Thus, if $w \in \ker(\beta_i\brq)$, then $cw^{q_0} \in
\im(\beta_{i+1}^{[qq_0]}) + IP_i$, and $c_1I=0$, so $dw^{q_0} =
c_1cw^{q_0} \in \im(\beta_{i+1}^{[qq_0]})$.

We next wish to show that the claim above is true for any finitely
generated module $M$.  We can induce on $\dim M$, with the case
$\dim M = 0$ done.  Assume that $\dim M>0$.  The same mapping cone
argument as above applied to $0 \to H_\m^0(M) \to M \to
M/H_\m^0(M) \to 0$ shows that we may assume that $\depth_\m M>0$.
Let $x \in \m$ be a nonzerodivisor on $M$ (so $\dim(M/xM)< \dim
M $). The mapping cone argument for the short exact sequence $0
\to M \overset x{\to} M \to M/xM \to 0$ gives an exact sequence
$H_i(F^e(P_\dt)) \overset {x^q}{\to}H_i(F^e(P_\dt)) \to
H_i(F^e(T_\dt))$.  Let $w \in \ker(\beta_i\brq)$.  Its image in
$F^e(T_\dt)$ gives a homology element, which by the induction
argument, vanishes when raised to the $q_0$ power and multiplied
by $d$.  Equivalently, $dw^{q_0} \in \im(\beta_{i+1}^{[qq_0]}) +
x^{qq_0}P_i$. However, the argument still applies for any power of
$x$, so by the Krull intersection theorem, $dw^{q_0} \in
\im(\beta_{i+1}^{[qq_0]})$, as desired.

A consequence of the above argument is that if $M$ is any finitely
generated $R$-module, and $(P_\dt,\beta_\dt)$ is a resolution of
$M$, then  $P_\dt \otimes_R R^{eq}$ is stably phantom at the $i$th
spot. But for any $j \ge 0$, the homology at the $i+j$th spot is
the $i$th homology of the $j$th syzygy of $M$, and therefore,
$P_\dt \otimes_R R^{eq}$ is stably phantom at all spots above $i$
as well. By Theorem 2.1.7 of \cite{Ab1}, $(\syz_{i-1} M)/I(\syz_{i-1} M)$ has  a
finite phantom resolution over $R^{eq}$, and therefore, if we take
$P_\dt$ to be minimal (so $P_\dt \otimes R^{eq}$ is minimal), we
see that $P_\dt$ is bounded on the left. But this means $\pd_R
\left(\syz_{i-1} M\right) < \8$, and hence $\pd_R M < \8$.  Hence $R$ is
regular.
\end{proof}

\begin{cor}\label{mainthm}
Let $R$ be a local ring of characteristic $p$.  If $\gothic t_i(R)
= 0$ for some $i>0$ then $R$ is a regular local ring.
\end{cor}

\begin{proof}
The hypothesis is stable under completion, and $R$ is regular if and
only if $\hat R$ is, so we may assume that $R$ is complete.

By Proposition~\ref{compareReq}, $G_\dt \otimes R^{eq}$ has stably
phantom homology at the $i$th spot.  Thus, by Proposition~\ref{mainprop},
$R$ is regular.
\end{proof}

\begin{rmk} Regarding Theorem \ref{mainprop}, one should not
expect more generally that if a resolution of an arbitrary module
$M$ has stably phantom homology at the $i$th spot for some $i>0$,
then $\pd M<\8$, although this is the case when $M$ is of finite
length over a local complete intersection \cite{M}. The reader
should refer to \cite{LM} for an explicit example.
\end{rmk}

As an immediate consequence of \ref{mainprop}, we have the
following improved version of a theorem due to Bridgeland and
Iyengar in the characteristic $p$ case,
\begin{thm}\label{BI}
Let $(R,\m,k)$ be a $d$-dimensional local ring in characteristic
$p>0$. Assume $C_\dt$ is a complex of free $R$-modules with
$C_i=0$ for $i \notin [0, d]$, the $R$-module $H_0(C_\bullet)$ is
finitely generated, and $\length(H_i(C_\bullet))<\infty$ for
$i>0$. If $k$ or any syzygy of $k$ is a direct summand of
$H_0(C_\bullet)$, then $R$ is regular.
\end{thm}
\begin{proof}The proof is the same as that of Theorem 3.1 in
\cite{Li}
\end{proof}
Since the original theorem of Bridgeland and Iyengar holds true in
the equicharacteristic case, we expect Theorem \ref{BI} also holds
in equicharacteristic case. But we do not have a proof.

As another corollary of Theorem \ref{mainprop}, we may also
characterize regularity for excellent local domains via vanishing
of higher Tor's of the residue field with $R^+$.

\begin{cor}\label{rplus} Let $(R,\m,k)$ be an excellent domain.  If $\Tor_i^R(k, R^+)
= 0$ for some $i>0$ then $R$ is regular.
\end{cor}
\begin{proof}
Let $(G_\dt,\al_\dt)$ be a minimal free resolution of $k$.  It
suffices to show that $G_\dt$ is stably phantom at the $i$th spot,
so that Theorem~\ref{mainprop} applies.

Suppose $w \in \ker(\al_i\brq)$.  Then $w^{1/q} \in \ker(\al_i\otimes_R R^+)
= \im(\al_{i+1} \otimes_R R^+)$ (the second equality holds since
$\Tor_i^R(k, R^+)= 0$).  This can be expressed using only finitely
many elements of $R^+$, so there is a module finite extension $S \supseteq R$
such that $w \in \im(\al_{i+1}^{[q]}\otimes_R S) \cap G_i
\inc \im(\al_{i+1}^{[q]})^*_{G_i}$, as desired.
\end{proof}
\begin{rmk} \label{yao} By exactly the same argument, one can also show that for a reduced, excellent and equidimensional local ring $R$, if
$\Tor_i^R(k,R^\8)=0$ for some $i>0$, then $R$ is regular, where
$R^\8$ denotes $\cup_q R^{1/q}$.
\end{rmk}
\specialsection*{ACKNOWLEDGEMENT} We thank Yongwei Yao for
pointing out Remark \ref{yao}. We also thank the referee for the
thoughtful comments.

\end{document}